\documentclass[12pt,a4paper]{amsart}
\usepackage{a4wide}
\usepackage[utf8]{inputenc}
\usepackage{xcolor}
\usepackage{bbold}
\usepackage{graphicx}
\usepackage{dsfont}
\usepackage[pagebackref, colorlinks = true, linkcolor = teal, urlcolor  = teal, citecolor = violet]{hyperref}
\usepackage{enumitem}

\newtheorem{thm}{Theorem}[section]
\newtheorem{lem}[thm]{Lemma}
\newtheorem{prop}[thm]{Proposition}
\newtheorem{rem}[thm]{Remark}

\newcommand{\tr}{\operatorname{tr}}
\newcommand{\id}{{\rm Id}}

\newcommand{\one}{{\mathbf 1}}
\newcommand{\Lip}{{\rm Lip}}
\renewcommand{\d}{{\rm d}}

\newcommand{\pp}{{\mathbb P}}
\newcommand{\ee}{{\mathbb E}}

\newcommand{\rr}{{\mathbb R}}
\newcommand{\nn}{{\mathbb N}}
\newcommand{\cc}{{\mathbb C}}
\renewcommand{\P}{{\mathrm P}}
\newcommand{\proj}{\pi}

\newcommand{\dist}{d}
\newcommand{\as}{\operatorname{-}\mathrm{a.s.}}
\newcommand\pt[1]{^{\otimes #1}}

\newcommand{\sta}{{\P(\cc^d)}}
\newcommand{\staalg}{\mathcal{B}}
\newcommand{\out}{\Omega}
\newcommand{\outalg}{\mathcal{O}}
\newcommand{\joint}{\sta\times \out}
\newcommand{\jointalg}{\mathcal{J}}

\begin{document}
\title{Limit theorems for Quantum Trajectories}
\author{Tristan Benoist}
\email{tristan.benoist@math.univ-toulouse.fr}
\address{Institut de Mathématiques, UMR5219, Université de Toulouse, CNRS, UPS, F-31062 Toulouse Cedex 9, France}
 
\author{Jan-Luka Fatras}
\email{jan-luka.fatras@ens-paris-saclay.fr}
\address{ENS Paris Saclay}

\author{Clément Pellegrini}
\email{clement.pellegrini@math.univ-toulouse.fr}
\address{Institut de Mathématiques, UMR5219, Université de Toulouse, CNRS, UPS, F-31062 Toulouse Cedex 9, France}

\maketitle

\begin{abstract}
   Quantum trajectories are Markov processes modeling the evolution of a quantum system subjected to repeated independent measurements. Under purification and irreducibility assumptions, these Markov processes admit a unique invariant measure -- see Benoist \emph{et al.} Probab. Theory Relat. Fields 2019. In this article we prove finer limit theorems such as Law of Large Numbers (LLN), Functional Central Limit Theorem, Law of Iterated Logarithm and Moderate Deviation Principle. The proof of the LLN is based on Birkhoff's ergodic theorem and an analysis of harmonic functions. The other theorems are proved using martingale approximation of empirical sums.
\end{abstract}

\tableofcontents

\section{Introduction}

Quantum trajectories are Markov processes modeling the evolution of a quantum system subjected to repeated independent measurements. A quantum system interacts with a series of independent identical probes. After each interaction a measurement is performed on the probe that just interacted.  Following the postulates of quantum mechanics, the measurement outcome and the resulting system state are random. Since the probes are independent and identical, the evolution of the system state is described by a homogeneous Markov chain. A typical experiment modeled by quantum trajectories is that of Serge Haroche's group (see \cite{guerlin_progressive_2007} for example). In this experiment, the system is a monochromatic electromagnetic field similar to a harmonic oscillator and the probes are atoms with two energy levels relevant for the experiment. The light field is measured and manipulated using the atoms and the measurements performed on them. Numerous quantum optics experiments are modeled by quantum trajectories -- see \cite{carmichael,haroche2006exploring,wisemanmilburn}.
 
From a mathematical point of view, these Markov chains are relatively singular. They are not generically $\varphi$-irreducible\footnote{The definition of $\varphi$-irreducibility is a standard generalization of the irreducibility criteria for countable state space Markov chains -- see \cite[\S 1.3.1]{MT}.} as we show in Section~\ref{KSKS}. This is not a complete surprise since they can be seen as place dependent iterated function systems (IFS) -- see the definition in \cite{barnsley_invariant_1988}. However, they do not satisfy the standard criteria of contractivity for IFS (see Section~\ref{KSKS}). Hence, standard results on uniqueness of invariant measures such as the one of \cite{barnsley_invariant_1988} cannot be applied. In \cite{benoist2019invariant} an approach based on the convergence of an estimator of the system state given the measurement outcomes was developed. Under common assumptions for quantum trajectories, the authors showed the Markov chain accepts a unique invariant measure and the convergence is exponentially fast towards this measure. The first assumption is the purification one, first introduced in \cite{Maassen}. In this seminal reference the authors showed that quantum trajectories tend to get closer to the set of extreme points (pure states) of the convex set they live in. They moreover showed that under the purification assumption, in large time limit, the Markov chain lives in this set. The second assumption concerns the average evolution of the system state. It is similar to the irreducibility assumption for finite state space Markov chains and leads to the ergodicity of the dynamical system formed by the sequence of measurement results (see \cite{KuMa1}). As mentioned in \cite{szczepanek2021ergodicity}, the probability measures describing the law of the measurement outcomes, are also known as Kusuoka measures as defined in \cite{kusuoka_dirichlet_1989}.

Since the pioneering works of Kümmerer and Maassen \cite{KuMa1,KuMa,Maassen} a large body of literature has been dedicated to limit theorems for the measurement outcomes.
Notably, Law of Large Numbers \cite{KuMa1}, Central Limit Theorems \cite{attal2015central,van2015sanov,carbone2015homogeneous} and Large Deviation Principles \cite{van2015sanov,carbone2015homogeneous,cuneo2019large} have been derived. Recently, extensions to non ergodic repeated quantum measurements have been made, see \cite{carbone2022generalized} and \cite{BePelLin} for a review of results and proof techniques, and for some extensions. The subject is still in development -- see \cite{benoist2018entropy,benoist2021entropy,girotti2022concentration, movassagh_ergodic_2021} for few examples. Physically relevant and mathematically pertinent examples of the rich behaviour measurement outcomes can exhibit can be found in \cite{benoist2021entropy}.

In the present article, we prove some limit theorems, not for the measurement outcomes, but for the system state, \emph{i.e.} the quantum trajectory $(\hat x_n)_{n}$. We study the time asymptotics of the empirical sum
$$S_n(g):=\sum_{k=0}^{n-1}g(\hat x_k)$$
for some function $g$.

We show a strong Law of Large Numbers (LLN), a Central Limit Theorem (CLT) and its functional version (FCLT), a Law of Iterated Logarithm (LIL) and a Moderate Deviation Principle (MDP). The proof of the LLN is a consequence of Birkhoff's stationary process ergodic theorem and the approximation of $\hat x_n$ by a random variable depending only on measurement outcomes established in \cite{benoist2019invariant}. Note that quantum trajectories are generally not positive Harris as required in \cite[Theorem~17.1.7]{MT}. Hence, not all bounded harmonic functions are constant. Only continuous harmonic functions are constant. We circumvent this issue by proving the probability the LLN holds is continuous with respect to the initial state of the system. The proofs of the other theorems are all based on a martingale approximation of $S_n(g)$ for $g$ Hölder continuous. The FCLT is derived using its standard version for martingales (see \cite[Theorem~4.1]{HallHeyde} for the version we use), the LIL proof follows from the original article \cite{stout1970martingale} and the MDP one is a consequence of \cite[Theorem~1.1]{Ga96}. The martingale approximation is a versatile tool. For example, a concentration bound similar to the one of \cite[Theorem~5]{girotti2022concentration} could also be derived. However, to our knowledge, it cannot be used to derive a large deviation principle. This question is still open for quantum trajectories.

\bigskip
We work in the same mathematical context as in \cite{benoist2019invariant} and use similar notations and definitions. The Hilbert space we consider is the canonical complex $d$ dimensional space $\mathbb{C}^d$ and especially its projective space $\P(\mathbb{C}^d)$ equipped with its Borel $\sigma$-algebra $\staalg$. 
For a non zero vector $x\in\mathbb C^d$, we denote $\hat x$ the equivalence class of $x$ in $\P(\mathbb C^d)$. For $\hat x\in \P(\cc^d)$, $x\in \cc^d$ is an arbitrary norm one representative of $\hat x$ and $\pi_{\hat x}$ is the orthogonal projector onto $\cc x$.
For a linear map  $v\in M_d(\mathbb C)$ we denote $v\cdot \hat{x} $ the element of the projective space represented by $v\,x$ whenever $v\,x\neq 0$. We equip $M_d(\mathbb C)$ with its Borel $\sigma$-algebra. We consider a measure $\mu$ on $M_d(\mathbb C)$ such that $v\mapsto \|v\|$ is square integrable, $$\int_{M_d(\mathbb C)} \|v\|^2\,\d\mu(v)<\infty,$$ and the following stochasticity condition is fulfilled,
\begin{equation}
\label{eq:stochastic family}
\int_{M_d(\mathbb C)} v^* v \,\mathrm{d} \mu(v) = \id_{\cc^d}. 
\end{equation}

A quantum trajectory is a realization of a Markov chain $(\hat{x}_n)$ on $\P(\mathbb{C}^d)$, defined by 
\[
\hat x_{n+1}=V_n\cdot \hat x_{n},
\]
where $V_n$ is a $M_d(\mathbb C)$-valued random variable with law $ ||v x_n||^2 \d \mu(v).$ More precisely, we study Markov chains associated with the transition kernel given for a set $S\in\staalg$ and $\hat x\in \P(\mathbb C^d)$ by
\begin{equation} \label{eq_deftranskernel}
\Pi(\hat x,S)=\int_{M_d(\cc)} \mathbf{1}_{S}(v\cdot \hat x)  \|v x\|^2 \d \mu(v).
\end{equation}
Every $\hat x_n$ can be written using a random product of matrices 
$$\hat x_n = V_{n}\ldots V_{1}\cdot \hat x_0,$$
with $(V_{n})_n$ non i.i.d.~but its law given by $\|v_n\dotsb v_1 x_0\|^2 \d\mu^{\otimes n}(v_1,\dotsc,v_n)$. Since the matrices are not i.i.d.~the results on products of i.i.d.~matrices (see \cite{Bougerol}) cannot be applied. Though, since $V_1x_0\neq 0$ almost surely, $V_1\cdot\hat x_0$ is almost surely well defined and we do not require the matrices be invertible. In \cite{benoist2019invariant} the authors showed that, assuming purification and matrix irreducibility, these Markov chains admit unique invariant measures $\nu_{inv}$. The matrix irreducibility assumption is weaker than the strong irreducibility one used for products of i.i.d. matrices. It only requires that there is a unique non trivial minimal common invariant subspace to the matrices charged by $\mu$. In the same reference, the convergence towards the invariant measure is shown to be geometric in Wasserstein distance uniformly with respect to the initial measure. This is our starting point. Indeed, such strong convergence result often leads to finer limit theorems for the chain. However, since $(\hat x_n)$ is in general not $\nu_{inv}$-irreducible (see Section~\ref{KSKS}), the usual proofs require adaptations.

First, in Section~\ref{LLNN}, we show a LLN for any continuous function $g$ on $\P(\mathbb C^d)$:
$$\frac 1n \sum_{k=0}^{n-1}g(\hat x_k)\mathop{\longrightarrow}_{n\to\infty} \ee_{\nu_{inv}}(g),\quad \text{a.s.}$$
whatever is the distribution $\nu$ of the initial state $\hat x_0$. This result cannot be extended to $L^1$ functions since, following \cite[Theorem~17.1.7]{MT}, it would imply quantum trajectories are positive Harris and therefore $\nu_{inv}$-irreduccible (see \cite[\S 9]{MT}) and, as already mentioned,  we provide a counter example to that property in Section~\ref{KSKS}. This result is a generalization of \cite[Theorem~5]{KuMa} which applies only to affine functions in $\pi_{\hat x}$. 

Second, in Section~\ref{Poisson}, we show that for any Hölder continuous function $g$, the Poisson equation admits a continuous solution $\tilde g$, that is
\begin{equation}\label{eq:Poisson eq}
    (I-\Pi) \tilde{g}=g -\ee_{\nu_{inv}}(g)=:\overline{g}.
\end{equation}
 The continuity of the solution is important since the LLN holds only for continuous functions. Then, as usual, the solution $\tilde{g}$ allows for the approximation of $$S_n(g):=\sum_{k=0}^{n-1}g(\hat x_k)$$
by a martingale $M_n(g)$ with uniformly bounded increments. This leads to several limit theorems inherited from limit theorems for martingales. In particular, in Section~\ref{sec:CLT} we show a CLT: 
$$\lim_{n\to\infty}\frac{S_n(\overline{g})}{\sqrt n}=\mathcal N(0,\gamma_g^2),\quad\mbox{in law, with }\gamma_g^2=\ee_{\nu_{inv}}(\tilde{g}^2-(\Pi \tilde{g})^2)$$
 and its functional version. In Section~\ref{LILL} we show a LIL:
 $$\limsup_{n\to\infty} \frac{\pm S_n(\overline{g})}{\sqrt{2n\log\log(n)}}=\gamma_g,\quad \mbox{a.s.}$$
 In Section~\ref{MDPP}, we show a MDP: for any function $n\mapsto a(n)$ such that $\lim_{n\to\infty} \frac{a(n)}{n}=0$ and $\lim_{n\to\infty}\frac{n}{a(n)^2}=0$ and $z\in \rr$,
 $$\lim_{n\to\infty}\sup_{\nu}\left|\frac{n}{a(n)^2}\log \ee_{\nu}\left(\exp\left(\tfrac{a(n)}{n}zS_n(\overline{g})\right)\right)-\tfrac12 z^2\gamma_g^2\right|=0,$$
 with the supremum taken over the distribution of $\hat x_0$. Moreover, if $\gamma_g^2>0$, for any Borel set $B\subset \rr$,
 \begin{align*}
\lim_{n\to\infty}\sup_{\nu}\frac{n}{a(n)^2}\log \pp_\nu(\tfrac{1}{a(n)}S_n(\overline{g})\in B)= -\inf_{y\in B}\frac{y^2}{2\gamma_g^2}.
 \end{align*}
 If $\gamma_g^2=0$, if $0$ is not in the closure of $B$,
 $$\lim_{n\to\infty}\sup_{\nu}\frac{n}{a(n)^2}\log \pp_\nu(\tfrac{1}{a(n)}S_n(\overline{g})\in B)=-\infty,$$
 if $0$ is in the interior of $B$,
 $$\lim_{n\to\infty}\sup_{\nu}\frac{n}{a(n)^2}\log \pp_\nu(\tfrac{1}{a(n)}S_n(\overline{g})\in B)=0.$$
 
 Notice that since $M_n(g)$ has uniformly bounded increments, the LLN could be proved using the martingale LLN and the density of Hölder continuous functions in the Banach space of continuous functions equipped with the supremum norm. However, since in general the existence of the solution to Poisson equation is not needed for the proof of the LLN, we preferred the proof presented here. Indeed, it may turnout useful in situations where the existence of a solution to Poisson equation cannot be proved easily, like when $\cc^d$ is replaced by an infinite dimensional Hilbert space.

\medskip
 There exists a continuous time version of quantum trajectories (see \cite{barchielli2009quantum, bouten2007introduction}). As shown in \cite{pellegrini2008existence,pellegrini2010markov,BaBeBe2}, they are approximations of the discrete time version we study here. Adapting our proofs, especially using the methods and results of \cite{benoist2021invariant} instead of \cite{benoist2019invariant} and up to possible standard technical adaptations, all our results translate to continuous time quantum trajectories.

 \medskip
Before we state and prove the main results, in Section~\ref{sec:notation}, we make precise the notation we use and recall some useful properties of quantum trajectories. At the end, in Section~\ref{KSKS}, we present an example of quantum trajectory to which our results apply yet the Markov chain does not verify the usual $\varphi$-irreducibility or contractivity assumptions.
\section{Notations and preliminaries}\label{sec:notation}

\subsection{Measurement process and Quantum trajectory}

In this article, we mostly follow the notations and definitions of \cite{benoist2019invariant}. We refer the reader to it for more details and comments. 

Let $\out:=M_d(\cc)^{\mathbb N}$. We equip it with its cylinder set $\sigma$-algebra. More precisely, let $\mathcal M$ be the Borel $\sigma$-algebra on $M_d(\cc)$. For $n\in\nn$, let $\outalg_n$ be the smallest $\sigma$-algebra on $\Omega$ making the sets $\{(v_1,v_2,\dotsc)\in \Omega: v_1\in A_1,\dotsc, v_n\in A_n\}$ measurable for any $A_1,\dotsc, A_n\in \mathcal M$. Then the smallest $\sigma$-algebra $\outalg$ containing $\outalg_n$ for all $n\in \nn$ makes $\Omega$ measurable. Let $\mathcal B$ be the Borel $\sigma$-algebra on $\P(\mathbb C^d)$, and denote
\[\jointalg_n=\mathcal B\otimes \outalg_n,\qquad \jointalg=\mathcal B\otimes \outalg.\]
This makes $\big(\P(\mathbb C^d)\times \Omega,\jointalg\big)$ a measurable space. We identify sub-$\sigma$-algebra $\{\emptyset,\P(\mathbb C^d)\}\times \outalg$ with $\outalg$, and equivalently identify any $\outalg$-measurable function $f$ with the $\jointalg$-measurable function $f$ satisfying $f(\hat{x},\omega) = f(\omega)$.

For $i\in\mathbb N$, we consider the random variables $V_i : \Omega \to  M_d(\mathbb C)$,
\begin{equation}
	V_i(\omega) = v_i \quad \mbox{for} \quad \omega=(v_1,v_2,\ldots), \label{eq:W}
\end{equation}
and we introduce $\mathcal O_n$-mesurable random variables $(W_n)$ defined for all $n\in\mathbb N$ as $$W_n=V_{n}\ldots V_{1}.$$ 

Let $\nu$ be a probability measure over $(\sta,\mathcal B)$. We extend it to a probability measure $\mathbb{P}_\nu$ over $(\sta\times\out,\jointalg)$ by setting, for any $S\in \staalg$ and any set $O_n \in \outalg_n$,
\begin{equation} \label{eq_defPnu}
\mathbb{P}_\nu(S \times O_n):=\int_{S\times O_n}  \|W_n(\omega)x\|^2 \d\nu(\hat x) \d\mu\pt n(\omega).
\end{equation}
Equation~\eqref{eq:stochastic family} ensures that~\eqref{eq_defPnu} defines a consistent family of probability measures and, by Kolmogorov's theorem, this defines a unique probability measure $\pp_\nu$ on $\joint$. In addition, the restriction of $\mathbb{P}_\nu$ to $\mathcal B\otimes \{\emptyset,\Omega\}$ is, by construction, $\nu$. In the sequel we shall refer systematically to $\mathbb P_\nu$ even if the involved r.v. does not imply $\mathcal O$ (this allows a unification of the notations all along the paper). 
\smallskip

We now define the random process $(\hat x_n)$. For $(\hat x, \omega)\in \joint$ let $\hat x_0(\hat x, \omega)=\hat x$. Note that for any $n$, the definition~\eqref{eq_defPnu} of $\pp_\nu$ imposes
\[\pp_\nu(W_n x_0 = 0)=0.\]
This allows us to define a sequence $(\hat x_n)$ of $(\jointalg_n)$-adapted random variables on the probability space $(\joint,\jointalg, \pp_\nu)$ by letting
\begin{equation} \label{eq_defxn}
\hat x_n:= W_n\cdot \hat x, 
\end{equation}
whenever the expression makes sense, i.e.\ for any $\omega$ such that $W_n(\omega) x\neq 0$, and extending it arbitrarily to the whole of $\Omega$. The process $(\hat x_n)$ on $(\Omega\times \P(\mathbb C^d),\jointalg, \mathbb{P}_\nu)$ has the same distribution as the Markov chain defined by $\Pi$ and initial probability measure $\nu$. We always work with this construction of the Markov chain. Namely, for any $n\in \nn$, $\hat x_n$ is a $\jointalg_n$-measurable random variable.

From a physical point of view, a sequence $\omega=(v_1,v_2,\ldots)$ is interpreted as the results of measurements that a physical apparatus records. Conditionally on these results, the Markov chain $(\hat x_n)$ is updated. Like hidden Markov chain models, $\omega=(v_1,v_2,\ldots)$ represents the observed process whereas $(\hat x_n)$ is not directly measured. Note that if $\mu$ is a sum of Dirac measures, that is for example
$$\mu=\sum_{i=1}^k\delta_{A_i},$$
the quantum trajectory is defined by
$$\hat x_{n+1}=A_i\cdot \hat x_n=\widehat{\frac{A_i x_n}{\Vert A_i x_n\Vert}}\,\,\textrm{with probability}\,\,\Vert A_i x_n\Vert^2.$$
These quantum trajectories are the most commonly encountered in the physics literature. A measurement apparatus produces a number $i\in\{1,\ldots,k\}$ as an outcome and the system state is updated conditionally to this information.

We focused our definitions on pure states, namely system states that can be represented by elements of $\P(\cc^d)$. Equivalently, these states are represented by rank one orthogonal projectors on $\cc^d$. It is often useful in quantum mechanics to consider a more general set of states, density matrices. They are elements of the convex hull of rank one projectors:
$$\mathcal D_d:=\operatorname{convex}\{\pi_{\hat x}:\hat x\in \P(\cc^d)\}=\{\rho\in M_d(\cc): \rho\geq 0, \tr\rho=1\}.$$
Quantum trajectory definition can be extended consistently to density matrices introducing the process $(\rho_n)$ defined by
$$\rho_{n+1}=\frac{V_{n+1} \rho_n V_{n+1}^*}{\tr(V_{n+1}\rho_n V_{n+1}^*)}\quad\mbox{with}\quad V_{n+1}\sim \tr(v^*v\rho_n)\d\mu(v)$$
conditioned on the value of $\rho_n\in \mathcal D_d$. However, since the set of pure states is stable under this random dynamics and our standing assumption \textbf{(Pur)} implies $\rho_n$ converges almost surely to the set of pure states as $n$ grows according to \cite{Maassen}, we limit ourselves to the pure state formulation of quantum trajectories.

Density matrices are useful to describe average states and restrictions of the measures $\pp_\nu$ to the outcome $\sigma$-algebra $\outalg$. To that end, for any probability measure $\nu$ over $\P(\cc^d)$, let
\begin{equation} \label{eq_defrhonu}
\rho_\nu:= \mathbb E_\nu(\proj_{\hat x}).
\end{equation}
By definition, $\rho_\nu\in\mathcal D_d$. 

We define probability measures over $\Omega$ that depend on a density matrix. For $\rho\in\mathcal{D}_d$ and any set $O_n \in \outalg_n$, let
\begin{equation} \label{eq_defPrho}
\pp^{\rho}(O_n):=  \int_{O_n}\tr\big(W_n(\omega) \rho \ W_n^*(\omega)\big) \mathrm{d} \mu\pt n(\omega). 
\end{equation}
Again using~\eqref{eq:stochastic family}, this defines a probability measure over $(\Omega,\outalg)$ through Kolmogorov extension theorem.

The measures $\pp_\nu$ can be expressed using the newly defined ones. For any $S\in\staalg$ and $A\in \outalg$,
\begin{equation}\label{eq:desintegration}
\pp_\nu(S\times A)=\int_{S}\pp^{\pi_{\hat x}}(A)\, \d\nu(\hat x).
\end{equation}
As proved in \cite[Proposition~2.1]{benoist2019invariant}, the marginal of $\mathbb{P}_\nu$ on $\outalg$ is the probability measure $\mathbb P^{\rho_\nu}$.

\medskip
We will manipulate some metric-related notion such as Hölder continuity. We thus equip $\P(\cc^d)$ with the metric
$$d(\hat x,\hat y)=\sqrt{1-|\langle x,y\rangle|^2}=\|\proj_{\hat x}-\proj_{\hat y}\|_\infty.$$

\subsection{Some preliminary results on $\Pi$}
The proof of geometric convergence towards the invariant measure in \cite{benoist2019invariant} relies on two assumptions:

\medskip
\noindent\textbf{Assumptions}\\
\noindent\textbf{(Pur)} Any orthogonal projector $\proj$ such that for any $n\in\nn$, $\proj v_1^*\ldots v_n^* v_n\ldots v_1 \proj \propto \proj$  for $\mu^{\otimes n}$-almost all $(v_1,\ldots,v_n)$, is of rank one.

\smallskip
\noindent \textbf{(Erg)} There exists a unique minimal subspace $E\neq\{0\}$ of $\cc^d$ such that $vE\subset E$ for all $v\in \operatorname{supp}\mu$.
\medskip

As mentioned in \cite{benoist2019invariant}, while formulated in terms of the measure $\mu$, assumption \textbf{(Erg)} depends only on the quantum channel
$$\phi:\rho\mapsto \int_{M_d(\mathbb C)} v\rho v^*\ \d\mu(v).$$
The assumption \textbf{(Erg)} is equivalent to $\phi$ accepting a unique fixed point $\rho_{inv}$ in $\mathcal D_d$. Then $E=(\ker \rho_{inv})^\perp$. Especially, assumption \textbf{(Erg)} with $E=\cc^d$ is equivalent to assuming $\phi$ is irreducible (see \cite[Section~6.2]{wolftour} for formulations of this notion of irreducibility).

Unless stated otherwise, all the results of the present article assume the same two assumptions hold. We do not mention them explicitly each time.
\medskip

We refer to \cite{benoist2019invariant} for a complete discussion regarding the link with the usual assumptions in the context of i.i.d product of random matrices (see \cite{Bougerol} for a reference on this subject). Let us mention that, as proved in \cite[Appendix~A]{benoist2019invariant}, \textbf{(Pur)} is equivalent to the contractivity assumption. Concerning \textbf{(Erg)}, as already mentioned in Introduction, it is weaker than the strong irreducibility condition. 

 The convergence obtained in \cite{benoist2019invariant} is expressed in terms of Wasserstein metric. A convergence in total variation is not accessible as is proved by the example of Section~\ref{KSKS}. Wasserstein metric of order $1$ between two probability measures $\sigma$ and $\tau$ on a metric space $(X,d)$ is defined by
 \begin{align}\label{eq:def W_1}
\mathcal W_1(\sigma,\tau)=\inf_{P}\int_{X\times X} d(x,y)\d P(x,y)
 \end{align}
 where the infimum is taken over all probability measures $P$ over $X\times X$ such that $P(\d x,X)=\sigma(\d x)$ and $P(X,\d y)=\tau(\d y)$.
 For $X$ compact, using Kantorovich--Rubinstein duality theorem it can be expressed as
$$\mathcal W_1(\sigma,\tau)=\sup_{f\in \Lip_1(X)}\left| \int_{X} f\,\d\sigma- \int _X f\, \d\tau\right|,$$
where $\Lip_1(X)=\{f:X\rightarrow\mathbb R \ \mathrm{s.t.}\ \vert f(x)-f(y)\vert\leq \dist(x,y)\}$ is the set of Lipschitz continuous functions with constant one.

Then, the main result of \cite{benoist2019invariant} is the following theorem, numbered Theorem~1.1 in the original reference.
\begin{thm}\label{thm:uniqueness}
The Markov kernel $\Pi$ accepts a unique invariant probability measure $\nu_{inv}$.

Moreover, there exists $m\in\{1,\ldots,d\}$, $C>0$ and $0<\lambda<1$ such that for any probability measure $\nu$ over $\big(\sta,\mathcal B\big)$,
\begin{equation}\label{eq:conv with period}
\mathcal W_1\left(\frac1m\sum_{r=0}^{m-1} \nu\Pi^{mn+r}, \nu_{inv}\right)\leq C \lambda^n.
\end{equation}
\end{thm}

We finish these preliminaries with a result that was overlooked in \cite{benoist2019invariant}. Indeed no proof of the existence of an invariant measure was provided. An invariant measure exists if $\Pi$ maps continuous functions to continuous function. In other words if $\Pi$ is Feller. As we will need this property, we provide a short proof. Note that the conclusion of the proposition holds even if \textbf{(Erg)} and \textbf{(Pur)} do not hold.
\begin{prop}\label{prop:feller}
The operator $\Pi$ is Feller, that is for every continuous function $f$, $\Pi f$ is continuous.
\end{prop}

\begin{proof}
Let $f$ be a continuous function on $\P(\mathbb C^d)$, then for all $\hat x$, we have
\begin{equation}
    \Pi f(\hat x)=\int_{M_d(\mathbb C)}f(v\cdot\hat x)\Vert vx\Vert^2\d\mu(v).
\end{equation}
We need to show that $\hat x\mapsto \Pi f(\hat x)$ is continuous.
Since $\P(\mathbb C^d)$ is compact, $f$ is bounded and
$$\psi:\hat x\mapsto f(v\cdot\hat x)\Vert vx\Vert^2$$
is continuous or can be extended to a continuous function whenever $vx=0$. The map $\psi$ is bounded by $\Vert f\Vert_\infty\Vert v\Vert^2$ which is independent of $\hat x$ and $\mu$ integrable since $v\mapsto \|v\|$ is assumed square integrable with respect to $\mu$. Lebesgue's dominated convergence theorem implies the continuity and the proposition is proved.
\end{proof}

\section{Law of Large Numbers}\label{LLNN}

This section is devoted to the Law of Large Numbers. We show the following theorem.

\begin{thm}\label{thm_LLN}
If the assumptions \textbf{(Erg)} and \textbf{(Pur)} hold, for every initial measure $\nu$ over $(\P(\cc^d),\staalg)$ and every continuous function $g$ on $\mathbb \P(\mathbb C^d)$,
$$\lim_{n\to\infty}\tfrac 1n S_n(g)=\lim_{n\to\infty}\frac 1n \sum_{k=0}^{n-1}g(\hat x_k)=\ee_{\nu_{inv}}(g), \quad \pp_\nu-a.s.$$
\end{thm}

As explained in Section~\ref{KSKS}, this theorem cannot be extended to every $L^1(\pp_{\nu_{inv}})$ function since it would imply the Markov chain $(\hat x_n)$ is $\nu_{inv}$-irreducible as a consequence of \cite[Theorem~17.1.7]{MT}. In Proposition~\ref{prop:no LLN KS} we provide an explicit example of $L^1$ function $g$ for which $S_n(g)=0$ for every $n\geq $ while $\nu_{inv}(g)=1$.
\subsection{Harmonic functions}

We recall that, in the context of Markov chains, harmonic functions are measurable functions $f$ such that $\Pi f=f$. They are deeply related to ergodic properties of Markov chains -- see \cite[\S 17.1]{MT}.

\begin{prop}\label{prop:cont_harmo_const}
    If the assumptions \textbf{(Erg)} and \textbf{(Pur)} hold, all continuous harmonic functions for $\Pi$ are constant.
\end{prop}
\begin{proof}
Let $\hat x\in \sta$, introduce for all $n\in\mathbb N^*$
$$\overline{\nu}_n(\hat x)=\frac1n\sum_{k=0}^{n-1}\delta_{\hat x}\Pi^k.$$
Since $\P(\cc^d)$ is compact, the sequence of measures $\overline{\nu}_n(\hat x)$ is tight and by Prokhorov's theorem, there exists a weak limit $\overline{\nu}_\infty(\hat x)$ along a subsequence $(n_l)_l$. Since $\Pi f=f$ by assumption and $f$ is continuous, $\Pi f$ is continuous. The weak convergence implies $\lim_{l\to\infty} \overline{\nu}_{n_l}(\hat x)\Pi f=\overline{\nu}_\infty(\hat x)\Pi f$. Since $\overline{\nu}_n(\hat x)\Pi=\overline{\nu}_{n}(\hat x)-\frac{1}{n}(\delta_{\hat x}-\delta_{\hat x}\Pi^{n})$,  it follows $\overline{\nu}_\infty(\hat x)\Pi=\overline{\nu}_{\infty}(\hat x)$. Theorem~\ref{thm:uniqueness} implies $\overline{\nu}_\infty(\hat x)=\nu_{inv}$. If $f$ is harmonic, $f(\hat x)=\ee_{\overline{\nu}_n(\hat x)}(f)$ for any $n\in \nn$. If moreover $f$ is continuous, the weak convergence of $\overline{\nu}_n(\hat x)$ implies $f(\hat x)=\ee_{\overline{\nu}_\infty(\hat x)}(f)=\ee_{\nu_{inv}}(f)$. Hence $f$ is constant and the proposition is proved. 
\end{proof}

For any continuous function $g:\P(\mathbb C^d)\rightarrow \mathbb C$, let us introduce the function $g_\infty$ defined by
$$ g_{\infty}(\hat{x})=\mathbb{P}_{\hat{x}}\left( \left\{ \tfrac{1}{n} S_n(g) \xrightarrow[n \rightarrow \infty]{} \mathbb E_{\nu_{inv}}(g) \right\} \right),$$
for all $\hat x\in \P(\mathbb C^d).$
We prove that this function is constant by applying Proposition~\ref{prop:cont_harmo_const}. To this end, we need the two following lemmas which imply the continuity. Before expressing the lemmas let us introduce the maximum likelihood estimator of the initial state 
\begin{equation} \label{eq_defzn}
\hat z_{n}(\omega)=\mathop{\mathrm{argmax}}_{\hat x\in \sta}\,\|W_n x\|^2
\end{equation}
and its evolved version
\begin{equation} \label{eq_defyn}
\hat y_n=W_n\cdot\hat z_n.
\end{equation}
Proposition~3.5 in \cite{benoist2019invariant} (or Lemma~2.3 in the same reference) proves that $\hat y_n$ is a good estimator of $\hat x_n$. We use this property to replace $(\hat x_n)$ by $(\hat y_n)$ in $g_\infty$.
\begin{lem} \label{lem:g O measurable}
Let $g$ be a continuous function on $\sta$. For any $\hat x\in \sta$,
\begin{align}
    g_\infty(\hat x)=\mathbb{P}^{\pi_{\hat{x}}}\left( \left\{ \frac{1}{n} \sum_{k=0}^{n-1} g(\hat{y}_k) \xrightarrow[n \rightarrow \infty]{} \mathbb E_{\nu_{inv}}(g) \right\} \right).
\end{align}
\end{lem}
\begin{proof}
Proposition~3.5 in \cite{benoist2019invariant} implies
$$d(\hat x_n,\hat y_n)\xrightarrow[n \rightarrow \infty]{}0,\quad\mathbb P_{\hat x}\as$$
Since $g$ is a continuous function and $\sta$ is compact, uniform continuity implies
$$\frac 1n\sum_{k=0}^{n-1}\left(g(\hat x_k)-g(\hat y_k)\right)\xrightarrow[n \rightarrow \infty]{}0,\quad\mathbb P_{\hat x}\as$$
This way it is clear that $g_\infty$ satisfies for all $\hat x\in \sta$
$$g_{\infty}(\hat{x})=\mathbb{P}_{\hat{x}}\left( \left\{ \frac{1}{n} \sum_{k=0}^{n-1} g(\hat{y}_k) \xrightarrow[n \rightarrow \infty]{} \mathbb E_{\nu_{inv}}(g) \right\} \right).$$
The set $\left\{ \frac{1}{n} \sum_{k=0}^{n-1} g(\hat{y}_k) \xrightarrow[n \rightarrow \infty]{} \mathbb E_{\nu_{inv}}(g) \right\}$ is $\mathcal O$ mesurable then Equation~\eqref{eq:desintegration} yields the lemma by taking $\nu=\delta_{\hat x}$ and $S=\sta$.
\end{proof}

We will use the next lemma to prove the continuity of $g_\infty$.
\begin{lem}\label{dist_var_totale}
For all states $\rho,\sigma\in \mathcal D_d$, we have the following estimation in total variation distance
$$\|\mathbb P^\rho-\mathbb P^\sigma\|_{TV}:=\sup_{A \in \mathcal{O}} |\mathbb{P}^{\rho}(A) - \mathbb{P}^{\sigma}(A)|\leq\|\rho-\sigma\|_{1},$$
where $\|\cdot\|_1$ is the trace norm defined by $\|A\|_1=\tr(\sqrt{A^*A})$.
\end{lem}
\begin{proof}
Since sets $O$ such that there exists $n\in \nn$  such that $O\in \mathcal O_n$ form a generating family for $\outalg$ and since the upper bound is independent of $O$, following \cite[Theorem~0.7]{walters2000introduction}, it is sufficient to prove the inequality for such sets. Let thus $O_n\in \outalg_n$ for an arbitrary $n\in \nn$,
\begin{eqnarray*}\label{P_rho_1-lip}
    |\mathbb{P}^{\rho}(O_n)-\mathbb{P}^{\sigma}(O_n)| & =& \Big|\int_{O_n} \text{tr} \big(W_n \rho W_n^* \big) d\mu^{\otimes n} - \int_{O_n} \tr \big(W_n \sigma W_n^* \big) \d\mu^{\otimes n} \Big| \\
    & =& \Big| \tr \big( (\rho-\sigma) \int_{O_n} W_n^*W_n \d\mu^{\otimes n} \big) \Big| \\
    & \leq&\tr\Big(|\rho-\sigma|\int_{O_n} W_n^*W_n \d\mu^{\otimes n}\Big) \\
    & \leq& \|\rho -\sigma \|_1,
\end{eqnarray*}
where for the first inequality, we used that $|\tr(AB)|\leq \tr(|A|B)$ for any $B$ positive semidefinite and $A$ self adjoint with $|A|=\sqrt{A^*A}$, and for the second one, we used that $\int_{M_d(\cc)^n} W_n^*W_n\d\mu\pt{n}=\id_{\cc^d}$.
\end{proof}

\begin{prop}\label{prop:ginfiny}
The function $g_\infty$ is continuous and harmonic, hence constant. 
\end{prop}

\begin{proof}
From Proposition~\ref{prop:cont_harmo_const}, it is sufficient to prove $g_\infty$ is continuous and harmonic. Let $\hat x$ and $\hat y$ be elements of $\P(\mathbb C^d)$, using Lemma~\ref{lem:g O measurable} and the fact that $\hat y_n$ is $\outalg$-measurable,
\begin{align*}
\vert g_\infty(\hat x)&-g_\infty(\hat y)\vert
\\&=\left\vert\mathbb{P}^{\pi_{\hat{x}}}\left( \left\{ \frac{1}{n} \sum_{k=0}^{n-1} g(\hat{y}_k) \xrightarrow[n \rightarrow \infty]{} \mathbb E_{\nu_{inv}}(g) \right\} \right)-\mathbb{P}^{\pi_{\hat{y}}}\left( \left\{ \frac{1}{n} \sum_{k=0}^{n-1} g(\hat{y}_k) \xrightarrow[n \rightarrow \infty]{} \mathbb E_{\nu_{inv}}(g) \right\} \right)\right\vert\\
&\leq\|\mathbb{P}^{\pi_{\hat{x}}}-\mathbb{P}^{\pi_{\hat{y}}}\|_{TV}\\
&\leq\| \pi_{\hat{x}}-\pi_{\hat{y}}\Vert_1\\
&\leq2\,d(\hat x,\hat y).
\end{align*}
We used that for a rank $r$ matrix $A$, $\|A\|_1\leq r \|A\|_\infty$.
Hence $g_\infty$ is 2-Lipschitz, therefore continuous.

The harmonicity of $g_\infty$ is a classical result and one can follow the proof of \cite[Proposition~17.1.6]{MT}. For sake of completeness, we reproduce it here in our context. Let $\hat x\in\P(\mathbb C^d)$
\begin{align*}
\Pi g(\hat x)=&\mathbb E_{\hat x}\left (\pp_{\hat x_1}\left(\lim_{n\to\infty}\frac1n S_n(g)= \mathbb E_{\nu_{inv}}(g)\right)\right)\\
	=&\mathbb E_{\hat x}\left (\left.\pp_{\hat x}\left(\lim_{n\to\infty}\frac1n \sum_{k=0}^{n-1} g(\hat{x}_{k+1})= \mathbb E_{\nu_{inv}}(g)\right|\mathcal J_1\right)\right)\\
	=& \pp_{\hat x}\left(\lim_{n\to\infty}\left[\frac{n+1}{n}\frac1{n+1} S_{n+1}(g)-\frac1n g(\hat x_0)\right]= \mathbb E_{\nu_{inv}}(g)\right)\\
	=&g(\hat x).
\end{align*}
The passage from line 1 to line 2 uses the Markov property and the remainder uses the fact that $g$ is bounded as a continuous function on a compact space.
\end{proof}

We turn to the proof of the Law of Large Numbers.

\subsection{Proof of the Law of Large Numbers, Theorem~\ref{thm_LLN}}

The proof is split in two parts. First Theorem~\ref{thm_LLN} is proved with respect to $\pp_{\nu_{inv}}$ using Birkhoff's ergodic theorem and \cite[Proposition~3.5]{benoist2019invariant}. Then, the result is extended to any initial probability measure $\nu$ using $g_\infty$ as is done for positive Harris Markov chains in \cite[Proposition~17.1.6]{MT}.

Before we prove Theorem~\ref{thm_LLN} with respect to $\pp_{\nu_{inv}}$ we prove the invariance of this measure with respect to the shift
\begin{align*}
\Theta:\sta\times\out&\to \sta\times\out\\
    (\hat x,v_1,v_2,\dotsc)&\mapsto (v_1\cdot \hat x,v_2,v_3,\dotsc).
\end{align*}
\begin{lem}\label{lem:shift_inv}
    The probability measure $\pp_{\nu_{inv}}$ is $\Theta$-invariant.
\end{lem}
\begin{proof}
    Let $f:\sta\times\out\to\rr$ be continuous and bounded. Then,
    $$\ee_{\nu_{inv}}(f\circ\Theta)=\int f(v_1\cdot \hat x,v_2,\dotsc)\d\pp^{\pi_{\hat x}}(v_1,v_2,\dotsc)\d\nu_{inv}(\hat x).$$
    By definition $\d\pp^{\pi_{\hat x}}(v_1,v_2,\dotsc)=\|v_1x\|^2\d\pp^{\pi_{v_1\cdot \hat x}}(v_2,v_3,\dotsc)\d\mu(v_1)$, thus,
    $$\ee_{\nu_{inv}}(f\circ\Theta)=\int f(v_1\cdot \hat x,v_2,\dotsc)\d\pp^{\pi_{v_1\cdot\hat x}}(v_2,v_3,\dotsc)\|v_1x\|^2\d\mu(v_1)\d\nu_{inv}(\hat x).$$
    By invariance of $\nu_{inv}$ with respect to $\Pi$,
    $$\ee_{\nu_{inv}}(f\circ\Theta)=\int f(\hat x,v_2,\dotsc)\d\pp^{\pi_{\hat x}}(v_2,\dotsc)\d\nu_{inv}(\hat x).$$
    A change of variables yields $\ee_{\nu_{inv}}(f\circ\Theta)=\ee_{\nu_{inv}}(f)$.
\end{proof}

We can now prove the LLN with respect to $\pp_{\nu_{inv}}$.
\begin{prop}\label{prop_LLNinv}
Let $g$ be a continuous function on $\P(\mathbb C^d)$. The LLN is satisfied for the invariant measure $\nu_{inv}$, that is
$$\tfrac1n S_n(g)\mathop{\longrightarrow}_{n\to\infty} \ee_{\nu_{inv}}(g),\quad\pp_{\nu_{inv}}\as$$
\end{prop}

\begin{proof}
This proof is an adaptation of \cite[Section~17.1]{MT}.
By definition, $S_n(g)$ can be written
$$S_n(g)=\sum_{k=0}^{n-1} g\circ\zeta\circ\Theta^k(\hat x,v_1,v_2\dotsc)$$
with for all $\hat x$ and all $v=(v_1,v_2,\ldots)$
$$\zeta(\hat x,v_1,v_2,\dotsc)=\hat x.$$
Following Lemma~\ref{lem:shift_inv}, the probability measure $\pp_{\nu_{inv}}$ is $\Theta$-invariant. Hence, Birkhoff's ergodic theorem for stationary processes implies,
\[\lim_{n\to\infty}\tfrac1nS_n(g)=X_g\quad\pp_{\nu_{inv.}}\as\]
with $X_g$ a $\Theta$ invariant random variable.

Since $g$ is continuous (and therefore uniformly continuous), the convergence $d(\hat x_n,\hat y_n)\xrightarrow[n \rightarrow \infty]{}0$ $\pp_{\nu_{inv}}$-almost surely implies
\[\lim_{n\to\infty}\frac1n\sum_{k=0}^{n-1} g(\hat y_k)=X_g \quad \pp_{\nu_{inv}}\as\]
Since the sequence $(\hat y_n)_n$ is $\mathcal O$-measurable, so is $X_g$. For $\outalg$-measurable functions, $\Theta$ is the left shift $(v_1,v_2,\dotsc)\mapsto (v_2,v_3,\dotsc)$. Since $\pp^{\rho_{\nu_{inv}}}$ is left shift ergodic as a consequence of \cite[Proposition 3.4]{benoist2019invariant} (see also \cite[Section~5]{KuMa1}), $X_g$ is $\pp^{\rho_{\nu_{inv}}}\as$ constant. Then, from the definition of $\pp^{\rho_{\nu_{inv}}}$,
\begin{equation*}
\label{eq:LLN-for-inv}
X_g=\mathbb E^{\rho_{\nu_{inv}}}(X_g)=\ee_{\nu_{inv.}}(X_g)=\ee_{\nu_{inv.}}(g),\quad \pp_{\nu_{inv.}}\as
\end{equation*}
\end{proof}

Now, we have all the ingredients to prove Theorem~\ref{thm_LLN}

\begin{proof}[Proof of Theorem~\ref{thm_LLN}] Proposition~\ref{prop_LLNinv} shows that 
$$\mathbb{P}_{\nu_{inv}}\left( \left\{ \frac{1}{n} \sum_{k=0}^{n-1} g(\hat{x}_k) \xrightarrow[n \rightarrow \infty]{} \mathbb E_{\nu_{inv}}(g) \right\} \right)=1,$$
which can be written as
$$\int_{\P(\mathbb C^d)}g_{\infty}(\hat x)d\nu_{inv}(\hat x)=1.$$
Since the integrand is no greater than one, it means that it is equal to one. For $\mathbb P_{\nu_{inv}}$-almost every $\hat x$, $g_\infty(\hat x)=1$ and since Proposition~\ref{prop:ginfiny} implies $g_\infty$ is constant, $g_\infty(\hat x)=1$ for all $\hat x\in \P(\mathbb C^d)$ and the theorem is proved.
\end{proof}

\section{Poisson equation and martingale approximation}\label{Poisson}
The remaining  limit theorems are proved using the existence of a continuous solution $\tilde{g}$ to the Poisson equation
$$\tilde g - \Pi \tilde g=\overline{g}=g-\mathbb E_{\nu_{inv}}(g)$$
We prove such a solution exists as long as $g$ is $\alpha$-Hölder for some $\alpha\in\,]0,1]$ where $1$-Hölder means Lipschitz.

Once a continuous solution $\tilde g$ is found, $S_n(\overline{g})$ can be approximated by the $(\jointalg_n)$-martingale
\begin{equation}\label{eq:def_M_n}
M_n(g):=\sum_{k=1}^{n} \tilde g(\hat x_k)-\Pi \tilde g(\hat x_{k-1}).
\end{equation}
Indeed, 
$$S_n(\overline{g})=\tilde g(\hat x_0)- \tilde g(\hat x_n) + M_n(g).$$
Since $\tilde g$ is continuous and therefore uniformly bounded, $(M_n(g))_n$ has bounded increments and
\begin{equation}\label{eq:martingale approximation}
    |S_n(\overline{g})-M_n(g)|\leq \operatorname{osc}(\tilde g)\leq 2\|\tilde g\|_\infty,    
\end{equation}
almost surely, where $\operatorname{osc}$ is the oscillation defined by $\operatorname{osc}(f)=\sup_{\hat x,\hat y} |f(\hat x)-f(\hat y)|$.

First we prove a geometric convergence for Hölder functions which will be a key ingredient in the sequel.
\begin{lem}\label{ineg_hold_lemma}
There exists $0<\lambda<1$ such that for any $g:\sta\to\cc$ Hölder continuous, there exists $C>0$ such that for any probability measure $\nu$ and all $n\in\mathbb N$ we have
\begin{equation}\label{ineg_hold}
    \left| \frac{1}{m} \sum_{r=0}^{m-1}\mathbb{E}_{\nu}[  g(\hat{x}_{mn+r})] - \mathbb{E}_{\nu_{inv}}[ g(\hat{x})]\right|\leq C\lambda^{\alpha n}
\end{equation}
where $\alpha$ is the Hölder exponent of $g$.
\end{lem}

\begin{proof}
For a coupling $P_n$ between $\frac{1}{m} \sum_{r=0}^{m-1}\nu\Pi^{mn+r}$ and $\nu_{inv}$,
$$\frac{1}{m} \sum_{r=0}^{m-1}\mathbb{E}_{\nu}[  g(\hat{x}_{mn+r})] - \mathbb{E}_{\nu_{inv}}[ g(\hat{x})]=\ee_{P_n}(g(\hat x)-g(\hat y)).$$
Since $g$ is $\alpha$-Hölder continuous, there exists $C$ such that
$$|\ee_{P_n}(g(\hat x)-g(\hat y))|\leq C\ee_{P_n}(d(\hat x,\hat y)^\alpha).$$
Using that $x\mapsto x^\alpha$ is concave, Jensen inequality implies,
$$|\ee_{P_n}(g(\hat x)-g(\hat y))|\leq C\left(\ee_{P_n}(d(\hat x,\hat y))\right)^\alpha.$$
The coupling $P_n$ being arbitrary we can take an infimum over it, so that
$$\left|\frac{1}{m} \sum_{r=0}^{m-1}\mathbb{E}_{\nu}[  g(\hat{x}_{mn+r})] - \mathbb{E}_{\nu_{inv}}[ g(\hat{x})]\right|\leq CW_1^\alpha\left(\frac{1}{m} \sum_{r=0}^{m-1}\nu\Pi^{mn+r},\nu_{inv}\right).$$
Then Theorem~\ref{thm:uniqueness} yields the lemma.
\end{proof}

A major upside in the above lemma is the uniformity of the bound with respect to the initial measure. We often exploit this property in the sequel and in particular in the following proof of existence of a continuous solution to the Poisson equation~\eqref{eq:Poisson eq}. 

\begin{prop}\label{theorem_exist_sol_eq_Poisson}
Let $g:\sta\to\cc$ be Hölder continuous. If the assumptions \textbf{(Erg)} and \textbf{(Pur)} hold, then the Poisson equation 
\begin{equation}\label{eq:poisson1}
(\id-\Pi)\tilde g=\bar g,
\end{equation}
admits a continuous solution $\tilde g$. Two continuous solutions of the Poisson equation~\eqref{eq:poisson1} are equal up to an additive constant.

As a consequence the martingale approximation of Equation~\eqref{eq:martingale approximation} holds. That is,
\begin{equation}
    \sup_{n}|S_n(\overline{g})-M_n(g)|\leq \operatorname{osc}(\tilde g)\leq 2\|\tilde g\|_\infty  
\end{equation}
almost surely.
\end{prop}

\begin{proof}
Without loss of generality, we assume $g=\bar g$. Let $$f_n^{(r)}=\sum_{k=0}^{nm+r-1}\Pi^k g,$$ for $r \in \{0,...,m-1\}$. Proposition~\ref{prop:feller} implies $\Pi$ preserves the continuity, hence $f_n^{(r)}$ is continuous for all $n \in\nn$ and $r \in \{0,...,m-1\}$.

For all $\hat x$ and $k\geq 0$, denote $\nu_{k,\hat x}=\delta_{\hat x}\Pi^k$, then for  all $p\leq q$ and all $\hat x\in \P(\mathbb C^d)$
\begin{equation} \label{Cauchy_fin}
    \begin{split}
    \left|f_p^{(r)}(\hat{x}) - f_q^{(r)}(\hat{x})\right|&=\Big|\sum_{k=pm+r}^{qm+r-1} \Pi^k g(\hat{x})\Big| = \Big|\sum_{k=pm}^{qm-1} \Pi^k \Big( \Pi^r g(\hat{x})\Big)\Big|\\&\leq m\left\vert\sum_{k=p}^{q-1}\frac{1}{m} \sum_{l=0}^{m-1}\Pi^r\Pi^{km+l}g(\hat x)\right\vert = m\left|\sum_{k=p}^{q-1}\frac1m\sum_{l=0}^{m-1}\nu_{r,\hat x}\Pi^{km+l}g\right|\\
    & \leq m \sum_{k=p}^{q-1} C\lambda^{\alpha k} \leq \frac{mC}{1-\lambda^\alpha}\lambda^{\alpha p},
    \end{split}
\end{equation}
where the constants $C$ and $\lambda$ are the ones from Lemma~\ref{ineg_hold_lemma} and are thus uniform in $\nu_{r,\hat x}$

It follows the sequence $\{f_p^{(r)}\}_{p \geq 0}$ is uniformly Cauchy and converges therefore uniformly towards a continuous function $f_{\infty}^{(r)}$. The equality $(\id-\Pi)f_n^{(r)} = g - \Pi^{nm+r}g$ and Lemma~\ref{ineg_hold_lemma} imply the uniform convergence
\begin{equation} \label{CU1}
    \begin{split}
    \lim_{n\to\infty}(\id-\Pi)\Big(\frac{1}{m}\sum_{r=0}^{m-1}f_n^{(r)}\Big)  & =g-\lim_{n\to\infty}\frac{1}{m}\sum_{r=0}^{m-1}\Pi^{nm+r}g =g.
    \end{split}
\end{equation}
Since $(\id - \Pi)$ is a bounded operator with respect to the uniform norm,
$$(\id-\Pi)\left( \frac1m\sum_{r=0}^{m-1} f_\infty^{(r)}\right)=g.$$
Hence, $\tilde g=\frac1m\sum_{r=0}^{m-1} f_\infty^{(r)}$ is a solution to Poisson equation~\eqref{eq:poisson1}. As a finite sum of continuous functions, it is also continuous.

In order to show uniqueness up to an additive constant, let $\tilde{g}_1$ and $\tilde{g}_2$ be two continuous solutions of Equation~\eqref{eq:Poisson eq}, then $\tilde{g}_1-\tilde{g}_2$ satisfies $\Pi (\tilde{g}_1-\tilde g_2)=\tilde g_1-\tilde g_2$. Since $\tilde{g}_1-\tilde g_2$ is continuous and harmonic, it is constant by Proposition~\ref{prop:cont_harmo_const}.
\end{proof}

In the remaining sections we state some limit theorems that are consequences of the martingale approximation of Equation~\eqref{eq:martingale approximation}.

\section{Central Limit Theorem}\label{sec:CLT}
This section is devoted to the proof of a FCLT. We define a version of $S_n$ as a continuous function interpolating between the partial sums $S_{\lfloor nt\rfloor}$ for $t\in [0,1]$. For any Hölder continuous function $g$ and $t\in [0,1]$ let
\begin{equation}\label{eq:def_s_n}
         s_n(t) := S_{\lfloor nt \rfloor}(g) + (nt-\lfloor nt \rfloor)(S_{\lfloor nt \rfloor +1}(g) - S_{\lfloor nt\rfloor}(g)).
\end{equation}
We recall that $\gamma_g^2=\ee_{\nu_{inv}}(\tilde g^2-(\Pi\tilde g)^2)$ with $\tilde g$ a solution to the Poisson equation~\eqref{eq:Poisson eq}.

\begin{thm}\label{pre_TCL}
Let $g:\sta\to\cc$ be Hölder continuous and such that $\ee_{\nu_{inv}}(g)=0$. If the assumptions \textbf{(Erg)} and \textbf{(Pur)} hold, 
\begin{enumerate}[label=(\arabic*)]
    \item If  $\gamma_g^2 > 0 $, for any initial measure $\nu$,
$$\left(\frac{s_n(t)}{\sqrt{n \gamma_g^2}}\right)_{t \in[0,1]} \xrightarrow[n \rightarrow \infty]{\mathcal{L}} B\equiv (B_t)_{t\in [0,1]},$$
where $B$ is a unidimensional Brownian motion and the convergence is meant for the process.\label{it:FCLT}
    \item If $\gamma_g^2=0$, then 
    $$\sup_{t\in [0,1]}|s_n(t)|\leq 6\operatorname{osc}(\tilde g),\quad \pp_{\nu_{inv}}\as$$
    and for any measure $\nu$, for any $t\in [0,1]$,
    $$\lim_{n\to\infty}\left\|\frac{s_n(t)}{\sqrt{n}}\right\|_{L^2(\pp_\nu)} = 0. $$\label{it:0CLT}
\end{enumerate}
\end{thm}

\begin{proof}
First, assume $\gamma_g^2 > 0$. We apply the Central Limit Theorem for martingales. For $t\in [0,1]$ and any $n$, let
$$m_n(t):=M_{\lfloor nt\rfloor}(g)+(nt-\lfloor nt\rfloor)(M_{\lfloor nt\rfloor +1}(g)-M_{\lfloor nt\rfloor}(g)).$$
We first show $(m_n(t))_t$ converges to a Brownian motion. We prove the following two properties:
\begin{enumerate}[label=\alph*)]
    \item $\pp_\nu$ almost surely
    $$\frac{1}{n} \sum_{k=1}^n \mathbb{E}[(M_k(g)-M_{k-1}(g))^2|\jointalg_{k-1}] \xrightarrow[n \rightarrow \infty]{}\gamma^2_g.$$\label{it:lln_variance}
    \item For all $\varepsilon>0$, $\pp_\nu$ almost surely
    $$\frac{1}{n} \sum_{k=1}^n \mathbb{E}\Big[\Big(M_k(g)-M_{k-1}(g)\Big)^2 \mathds{1}_{\{(M_k(g)-M_{k-1}(g))^2 \geq \varepsilon n\}} | \jointalg_{k-1}\Big]\xrightarrow[n \rightarrow \infty]{}0.$$\label{it:tight_clt_martingale}
\end{enumerate}

Item~\ref{it:tight_clt_martingale} follows directly from $\sup_k\|M_k(g)-M_{k-1}(g)\|_\infty\leq 2\|\tilde g\|_\infty$. Indeed that implies the terms in the sum on the left hand side vanish for $n$ large enough. Hence the left hand side is null for $n\geq 4\|\tilde g\|_\infty^2/\varepsilon$.

Item~\ref{it:lln_variance} follows from the LLN. For any $k\in \nn$,
\begin{equation}\label{conv_gamma_g}
        \begin{split}
         \mathbb{E}[(M_k(g)-M_{k-1}(g))^2|\jointalg_{k-1}]& = \mathbb{E}[(\tilde g(\hat{x}_k)-\Pi \tilde g(\hat{x}_{k-1}))^2|\jointalg_{k-1}] \\
        & =\Big( \Pi {\tilde g}^2(\hat{x}_{k-1}) - (\Pi \tilde g(\hat{x}_{k-1}))^2\Big)\\& = h(\hat{x}_{k-1}),
        \end{split}
    \end{equation}
    where $h(\cdot) = \Pi \tilde g^2(\cdot) - (\Pi \tilde g(\cdot))^2$. Since $\Pi$ is Feller (see Proposition~\ref{prop:feller}) and $\tilde g$ is continuous, $h$ is a continuous function. Then, Theorem~\ref{thm_LLN} and $\gamma_g^2=\mathbb E_{\nu_{inv}}(h)$ yield Item~\ref{it:lln_variance}.
    
Theorem~4.1 in \cite{HallHeyde} implies
$$\left(\frac{m_n(t)}{\sqrt{n \gamma_g^2}}\right)_{t \geq 0} \xrightarrow[n \to \infty]{\mathcal{L}} B.$$
We turn to $(s_n(t))_t$. Proposition \ref{theorem_exist_sol_eq_Poisson} implies
\begin{equation}
        \sup_{t\in[0,1],n\in \nn}|m_n(t)-s_n(t)| \leq 6\operatorname{osc}(\tilde g).
\end{equation}
 Thus,
    \begin{equation}\label{s_n_et_m_n}
        \sup_{t\in [0,1]}\left| \frac{s_n(t)-m_n(t)}{\sqrt{n\gamma_g^2}}\right|  \xrightarrow[n \rightarrow \infty]{} 0,\quad \pp_\nu\as
\end{equation}
Then Slutsky's lemma implies Item~\ref{it:FCLT}.

   \medskip
Second, assume $\gamma_g^2=0$. By the martingale property, $\ee_{\nu_{inv}}(M_{\lfloor n t\rfloor}^2(g))=\sum_{k=1}^{\lfloor nt\rfloor}\ee_{\nu_{inv}}((\tilde g(\hat x_k)-\Pi \tilde g(\hat x_{k-1}))^2)$. Taking a conditional expectation with respect to $\jointalg_{k-1}$ inside the expectation leads to
$$\ee_{\nu_{inv}}(M_{\lfloor nt\rfloor}^2(g))=\sum_{k=1}^{\lfloor nt\rfloor} \ee_{\nu_{inv}}\left(h(\hat x_{k-1})\right).$$ 
Invariance of $\nu_{inv}$ implies,
$$\ee_{\nu_{inv}}(M_{\lfloor nt\rfloor}^2(g))=\lfloor nt\rfloor\gamma_g^2=0.$$
Hence, $M_{\lfloor nt\rfloor}(g)=0$ $\pp_{\nu_{inv}}$-almost surely. Similarly, $M_{\lfloor nt\rfloor+1}(g)=0$ $\pp_{\nu_{inv}}$-almost surely. Then,  Proposition \ref{theorem_exist_sol_eq_Poisson} yields the first part of Item~\ref{it:0CLT}.

For an arbitrary measure $\nu$ over $\sta$, the martingale property implies 
$$\ee_\nu(M_{\lfloor nt\rfloor}^2(g))=\sum_{k=1}^{\lfloor nt\rfloor}\nu(\Pi^{k-1}h).$$
Then, Theorem~\ref{thm:uniqueness} and Proposition~\ref{prop:feller} imply $\frac{1}{\lfloor nt\rfloor}\ee_{\nu}(M_{\lfloor nt\rfloor}^2(g))\xrightarrow[n\to\infty]{}\gamma_g^2=0$. Similarly, $\frac{1}{\lfloor nt\rfloor+1}\ee_{\nu}(M_{\lfloor nt\rfloor+1}^2(g))\xrightarrow[n\to\infty]{}\gamma_g^2=0$.  Then Proposition \ref{theorem_exist_sol_eq_Poisson} yields the $L^2$-norm convergence of the second part of Item~\ref{it:0CLT}.
\end{proof}

\begin{rem}
The convergence in $L^2$-norm when $\gamma_g^2=0$ should also hold $\pp_\nu$ almost surely. However, such a result requires a finer analysis of the function $h$. In \cite{Hautecoeur}, $h$ is shown to be Hölder continuous. That implies $\sup_n \ee_\nu(M_n^2(g))<\infty$ by Lemma~\ref{ineg_hold_lemma}. Then the almost sure convergence follows from Doob's martingale convergence theorem as in the proof of \cite[Theorem~17.1.7]{MT}.
\end{rem}

\begin{rem}
The case $\gamma_g^2=0$ is somewhat trivial. We therefore concentrate on the situation where $\gamma_g^2>0$.
\end{rem}

\section{Law of Iterated Logarithm}\label{LILL}

This section is devoted to the the Law of Iterated Logarithm. It shows that at only a slightly faster speed ($\log \log n$) the weak convergence of the CLT is lost.

\begin{thm}
Let $g:\sta\to\cc$ be Hölder continuous and such that $\ee_{\nu_{inv}}(g)=0$. Assume \textbf{(Erg)} and \textbf{(Pur)} hold. Let $\tilde g$ be a continuous solution to the Poisson equation~\eqref{eq:Poisson eq} and  $\gamma_g^2=\ee_{\nu_{inv}}(\tilde g^2 -(\Pi \tilde g)^2)$. Assume $\gamma_g^2 > 0$. Then for any initial measure $\nu$, $$\limsup\limits_{n \rightarrow \infty}\frac{\pm S_n(g)}{\sqrt{2n\gamma_g^2 \log\log(n)}} = 1 \;\; \mathbb P_{\nu} \as$$
\end{thm}

\begin{proof}
As for the FCLT, Proposition \ref{theorem_exist_sol_eq_Poisson} allows us to prove the statement only for $M_n(g)$ instead of $S_n(g)$. We apply the theorem of \cite{stout1970martingale}. 

Let 
$$s_n^2:=\sum_{k=1}^n \mathbb{E}_\nu[(M_k(g)-M_{k-1}(g))^2|\jointalg_{k-1}]$$
and 
$$u_n=\sqrt{2 \log\log(s_n^2)}.$$
Following~\eqref{conv_gamma_g}, since $\Pi$ is Feller according to Proposition~\ref{prop:feller}, $h$ is continuous and Theorem~\ref{thm_LLN} implies $\frac1n s_n^2\xrightarrow[n\to\infty]{}\gamma_g^2,\,\, \pp_\nu\as$. Hence,
$$s_n^2\xrightarrow[n\to\infty]{}+\infty,\,\,\pp_\nu\as$$
Then, since almost surely,
$$|M_n(g)-M_{n-1}(g)|\leq 2\|\tilde g\|_\infty,$$
following the remark at the end of \cite[Section 3]{stout1970martingale}, the theorem in the introduction of the same reference applies and
$$\limsup_{n\to\infty}\frac{\pm M_n(g)}{s_n u_n}=1,\quad \pp_\nu\as$$
Since $s_n^2/n\xrightarrow[n\to\infty]{}\gamma_g^2,\,\,\pp_\nu\as$, we have the almost sure equivalences $s_n\sim \sqrt{n\gamma_g^2}$ and $u_n\sim \sqrt{2\log\log n}$ as $n$ grows. These equivalences yield the theorem.
\end{proof}

\section{Moderate Deviation Principle}\label{MDPP}
Our final limit theorem concerns the probability of deviation from the law of large numbers for rates in $n$ faster that the $\sqrt{n}$ of the CLT and slower than the linear one of the LLN. The Moderate Deviation Principle shows that at any of these rates the probability of deviation is essentially gaussian with variance $\gamma_g^2$.
\begin{thm}\label{thm_moderate}
Let $g:\sta\to\cc$ be Hölder continuous and such that $\ee_{\nu_{inv}}(g)=0$. Let $t\mapsto a(t), t>0$ be a non negative function such that 
$$\lim \limits_{t \rightarrow \infty} \frac{t}{a(t)} = \lim \limits_{t \rightarrow \infty} \frac{a(t)}{\sqrt{t}} = +\infty.$$
Then, if the assumptions \textbf{(Erg)} and \textbf{(Pur)} hold, 
\begin{enumerate}[label=(\arabic*)]
    \item For all $z \in \mathbb R$, $$\lim \limits_{n \rightarrow \infty} \sup \limits_{\hat{x} \in P(\cc^d)} \Big| \frac{n}{a^2(n)} \log \Big( \mathbb{E}_{\hat{x}} \Big[ \exp{\frac{a(n)}{n} z S_n(g)} \Big] \Big) -\frac{z^2 \gamma_g^2}{2} \Big| = 0.$$\label{it:cumulant MDP}
    \item For any Borel subset $B$ of $\rr$:
    $$\limsup \limits_{n \rightarrow \infty} \frac{n}{a^2(n)} \log \Big( \sup \limits_{\hat{x} \in P(\cc^d)} \mathbb{P}_{\hat{x}} \Big( \tfrac{1}{a(n)} S_n(g) \in B \Big) \Big)\leq -\inf_{y \in \Bar{B}} J(y)$$
    $$\liminf \limits_{n \rightarrow \infty} \frac{n}{a^2(n)} \log \Big( \inf \limits_{\hat{x} \in P(\cc^d)} \mathbb{P}_{\hat{x}} \Big( \tfrac{1}{a(n)} S_n(g) \in B \Big) \Big) \geq -\inf_{y \in B^{\circ}} J(y)$$
    where $\bar B$ denotes the closure of $B$ and $B^{\circ}$ its interior and
    $$J(y) = \sup \limits_{v \in \mathbb R} \big( yv - \tfrac12\gamma_g^2 v^2 \big) = \begin{cases}\frac{y^2}{2\gamma_g^2} &\text{if }\gamma_g^2>0\\
    0& \text{if }\gamma_g^2=0\text{ and }y=0\\
    +\infty&\text{if }\gamma_g^2=0\text{ and }y\neq 0
    \end{cases}.$$
\end{enumerate}
\end{thm}

\begin{rem}
    If $\gamma_g^2>0$, $y\mapsto J(y)$ is continuous, then, for $B$ with non empty interior, the inequalities turn into equalities. Hence,
    $$\lim_{n\to\infty}\frac{n}{a^2(n)}\log\left(\sup_{\hat x\in \sta}\pp_{\hat x}\left(\tfrac{1}{a(n)}S_n(g)\in B\right)\right)=-\inf_{y\in B}\frac{y^2}{2\gamma_g^2}.$$
\end{rem}
We prove Theorem~\ref{thm_moderate} through an application of \cite[Theorem~1.1]{Ga96} to $(M_n(g))_n$. In our context the two assumptions of this theorem are:

(A.1) there exists $\delta>0$ such that,
$$\sup_{\hat x\in \sta}\sup_{n\geq 0} \|\mathbb E_{\hat x}(\exp(\delta|M_{n+1}(g)-M_n(g)|)|\mathcal F_n)\|_{L^\infty(\pp_{\hat x})}<+\infty,$$

(A.2) the following limit inferior holds:
$$\liminf_{n\to\infty,\frac{n}{m}\to0}\sup_{\hat x\in \P(\mathbb C^d)}\sup_{j\geq 0}\Big| \frac{1}{m} \mathbb{E}_{\hat{x}}\Big[ \sum_{i=1}^{m} \big( \tilde g(\hat{x}_{n+j+i}) - \Pi \tilde g (\hat{x}_{n+j+i-1})\big)^2\Big|\mathcal F_j \Big] - \gamma_g^2 \Big|=0.$$

Assumption~(A.1) is trivially satisfied with an arbitrary $\delta>0$ since $\sup_{n\in \nn}|M_{n+1}(g)-M_n(g)|\leq 2\|\tilde g\|_\infty$ $\pp_{\hat x}$-almost surely with $\tilde g$ a continuous solution to the Poisson equation~\eqref{eq:Poisson eq}. Assumption~(A.2) requires more work and cannot be derived from \cite{Ga96} since in that article the Markov chain is assumed Harris recurrent. The next lemma shows a uniform convergence of the Cesàro mean of $\ee_{\hat x}(h(\hat x_n))$ with $h$ defined in Equation~\eqref{conv_gamma_g} to $\gamma_g^2$. It is the key to the proof of Assumption~(A.2).

\begin{lem}\label{lem_uniform1} Let $g:\sta\to\cc$ and $\tilde g$ a continuous solution to the Poisson equation~\eqref{eq:Poisson eq}. Then,
\begin{equation*}
    \lim \limits_{n \rightarrow \infty} \sup \limits_{\hat{x} \in P(\cc^d)} \sup \limits_{j \geq 0} \left| \frac{1}{n} \mathbb{E}_{\hat{x}}\left[\sum_{i=1}^n \big(\tilde g(\hat{x}_{j+i})-\Pi \tilde g(\hat{x}_{j+i-1}) \big)^2 \right] - \gamma_g^2 \right| = 0.
\end{equation*}
\end{lem}

\begin{proof}
Taking the conditional expectation with respect to $\jointalg_{j+i-1}$,
$$\ee_{\hat x}(\tilde g(\hat x_{j+i})-\Pi\tilde g(\hat x_{i+j-1}))^2=\ee_{\hat x}(h(\hat x_{i+j-1}))$$
with $h:\hat x\mapsto \Pi\tilde g^2(\hat x)-(\Pi \tilde g(\hat x))^2$ the continuous function defined in Equation~\eqref{conv_gamma_g}.

Since Lipschitz continuous functions are dense in the set of continuous functions with respect to the sup norm, for any $\epsilon>0$ there exists $h_\epsilon$ Lipschitz continuous such that $\|h-h_\epsilon\|_\infty\leq \epsilon$. It follows that
$$|\gamma_g^2-\ee_{\nu_{inv}}(h_\epsilon)|\leq\epsilon$$
and
$$\sup_{n\in \nn}\sup \limits_{\hat{x} \in P(\cc^d)} \sup \limits_{j \geq 0} \left| \frac{1}{n} \mathbb{E}_{\hat{x}}\left[\sum_{i=1}^n h(\hat x_{i+j-1}) \right] - \frac{1}{n} \mathbb{E}_{\hat{x}}\left[\sum_{i=1}^n h_\epsilon(\hat x_{i+j-1}) \right]\right|\leq \epsilon.$$
Since $\epsilon$ is arbitrary and the second bound is uniform in $n\in \nn$, it remains only to prove
$$\lim \limits_{n \rightarrow \infty} \sup \limits_{\hat{x} \in P(\cc^d)} \sup \limits_{j \geq 0} \left| \frac{1}{n} \sum_{i=1}^{n}\mathbb{E}_{\hat{x}}\left[h_\epsilon(\hat{x}_{j+i-1}) \right] - \ee_{\nu_{inv}}(h_\epsilon) \right| = 0$$
for any $\epsilon>0$.

First, by the Markov property,
$$\sup \limits_{\hat{x} \in P(\cc^d)} \sup \limits_{j \geq 0} \left| \frac{1}{n} \sum_{i=1}^{n}\mathbb{E}_{\hat{x}}\left[h_\epsilon(\hat{x}_{j+i-1}) \right] - \ee_{\nu_{inv}}(h_\epsilon) \right|\leq \sup \limits_{\nu \in \mathcal P(\sta)}\left| \frac{1}{n} \sum_{i=1}^{n}\mathbb{E}_{\nu}\left[h_\epsilon(\hat{x}_{i-1}) \right] -  \ee_{\nu_{inv}}(h_\epsilon) \right|,$$
where $\mathcal P(\sta)$ is the set of probability measures over $\sta$. By definition of the Wasserstein distance of order $1$, Theorem~\ref{thm:uniqueness} implies there exists $C>0$ such that
$$\sup \limits_{\nu \in \mathcal P(\sta)}\left| \frac{1}{n} \sum_{i=1}^{n}\mathbb{E}_{\nu}\left[h_\epsilon(\hat{x}_{i-1}) \right] - \ee_{\nu_{inv}}(h_\epsilon) \right|\leq C\tfrac1n.$$
Taking $n$ to infinity yields the lemma.
\end{proof}

\begin{proof}[Proof of Theorem~\ref{thm_moderate}]
As remarked in \cite{Ga96}, it is sufficient to prove Item~\ref{it:cumulant MDP} since then Gärtner--Ellis theorem (see \cite[Theorem~2.3.6]{dembo2009large}) implies the remainder of the theorem. We adapt the proof of Theorem~1.3 in \cite{Ga96}.

As mentioned before Lemma~\ref{lem_uniform1}, it only remains to prove Assumption~(A.2) of Theorem~1.1 in \cite{Ga96}.
For $m> 0$,
\begin{equation*}
    \begin{split}
    \sup_{\hat x}\sup_{n\geq 0}\sup_{j\geq 0}\Big| \frac{1}{m} \mathbb{E}_{\hat{x}}\Big[ \sum_{i=1}^{m} \big( \tilde g(\hat{x}_{n+j+i})& - \Pi \tilde g (\hat{x}_{n+j+i-1})\big)^2\Big|\mathcal F_j \Big] - \gamma_g^2 \Big|\\
    &= \sup_{\hat x}\sup_{n\geq 0}\sup_{j\geq 0}\Big| \frac{1}{m} \mathbb{E}_{\hat{x}_j}\Big[ \sum_{i=1}^{m} \big( \tilde g(\hat{x}_{n+i}) - \Pi \tilde g (\hat{x}_{n+i-1})\big)^2 \Big] - \gamma_g^2 \Big| \\
    & \leq \sup \limits_{\hat{y}} \sup \limits_{n \geq 0} \Big| \frac{1}{m} \mathbb{E}_{\hat{y}}\Big[ \sum_{i=1}^{m} \big( \tilde g(\hat{x}_{n+i}) - \Pi \tilde g (\hat{x}_{n+i-1})\big)^2 \Big] - \gamma_g^2 \Big| \\
    \end{split}
\end{equation*}
Lemma~\ref{lem_uniform1} implies the convergence to $0$ of the first line. This convergence being uniform in $n$, it implies 
$$\liminf_{n\to\infty,\frac{n}{m}\to0}\sup_{\hat x\in \P(\mathbb C^d)}\sup_{j\geq 0}\Big| \frac{1}{m} \mathbb{E}_{\hat{x}}\Big[ \sum_{i=1}^{m} \big( \tilde g(\hat{x}_{n+j+i}) - \Pi \tilde g (\hat{x}_{n+j+i-1})\big)^2\Big|\mathcal F_j \Big] - \gamma_g^2 \Big|=0,$$
which is Assumption~(A.2) of Theorem 1.1 in \cite{Ga96}. It follows that for all $z \in \rr$,
$$\lim \limits_{n \to \infty} \sup \limits_{\hat{x} \in P(\cc^d)} \Big| \frac{n}{a^2(n)} \log \Big( \mathbb{E}_{\hat{x}} \Big[ \exp{\frac{a(n)}{n} z M_n(g)} \Big] \Big) -\frac{z^2 \gamma_g^2}{2} \Big| = 0.$$
Since $| S_n(g)-M_n(g)|\leq 2\operatorname{osc}(\tilde g)$ almost surely, Item~\ref{it:cumulant MDP} holds and the theorem is proved.
\end{proof}

\section{Keep-Switch example}\label{KSKS}
To illustrate our results and demonstrate quantum trajectories can be relatively singular, we briefly discuss the ``Keep-Switch'' model presented in \cite[Section~2.1.2]{benoist2021entropy}. Let $\mu = \delta_{A_1} + \delta_{A_2}$ be a measure on $M_2(\cc)$, where 
$$A_1=\begin{pmatrix}
  \sqrt{p} & 0\\
  0 & \sqrt{q}
\end{pmatrix}\,\,\textrm{and}\,\,A_2=\begin{pmatrix}
  0 & \sqrt{p}\\
  \sqrt{q} & 0
\end{pmatrix}$$
with $p\in ]0,\frac12[$ and $q=1-p$. Direct computation shows that $A_1^*A_1+A_2^*A_2=\id$. Since $A_1^*A_1=\begin{pmatrix}
    p&0\\ 0&q
\end{pmatrix}$ is not proportional to the identity, the purification assumption \textbf{(Pur)} holds. The matrices $A_1$ and $A_2$ have a common non trivial invariant subspace only if they have a common eigenvector. The eigenvectors of $A_1$ are multiples of $e_a=\begin{pmatrix}
    1\\0
\end{pmatrix}$ and $e_b=\begin{pmatrix}
    0\\1
\end{pmatrix}$. They are not eigenvectors of $A_2$. Hence, $\mu$ verifies the ergodicity assumption \textbf{(Erg)} with $E=\cc^2$.
Direct computation shows that the invariant measure of Theorem~\ref{thm:uniqueness} is $\nu_{inv}=p\delta_{\hat e_a}+q\delta_{\hat e_b}$.

This relatively simple example already exhibits some unusual properties. For example, it is not $\varphi$-irreducible as defined in \cite{MT}. It cannot therefore be Harris recurrent.
\begin{prop}
    There does not exist a non trivial positive measure $\varphi$ over $\sta$ such that the Markov chain associated to $\mu$ is $\varphi$-irreducible.
\end{prop}
\begin{proof}
    Assume such a measure $\varphi$ exists. Then $\Pi\one_{\{\hat e_a,\hat e_b\}}(\hat e_a)=1$ implies $\varphi(\sta\setminus\{\hat e_a,\hat e_b\})=0$. Since $\varphi$ is non trivial, $\varphi(\{\hat e_a,\hat e_b\})>0$. However, setting $e_+=e_a+e_b$, $\pp_{\hat e_+}(\cup_n\{\hat x_n\in \{\hat e_a,\hat e_b\}\})=0$ as a consequence of $\pp_{\hat e_+}(\cup_n\{\langle e_a,x_n\rangle=0\vee \langle e_b,x_n\rangle=0\})=0$. Hence, there does not exists $\varphi$ such that the Markov chain is $\varphi$-irreducible.
\end{proof}
As mentioned at the beginning of Section~\ref{LLNN}, since the Markov chain is not Harris recurrent, following \cite[Theorem 17.1.7]{MT}, there must exist a non continuous $L^1(\nu_{inv})$ function such that the LLN does not hold. 
\begin{prop}\label{prop:no LLN KS}
    Let $g=\one_{\{\hat e_a,\hat e_b\}}$, then 
    $$\mathbb E_{\nu_{inv}}(g)=1$$
    but for any $\hat x\in \sta\setminus\{\hat e_a,\hat e_b\}$, any $n\geq0$,
    $$S_n(g)=0,\quad \pp_{\hat x}\as$$
\end{prop}
\begin{proof}
    The fact that $\nu_{inv}(g)=1$ follows from the expression of the invariant measure.

    For the proof that $S_n(g)=0$, as in the proof of previous proposition, for any $\hat x\in \sta\setminus\{\hat e_a,\hat e_b\}$, $\pp_{\hat x}(\cup_n\{\hat x_n\in \{\hat e_a,\hat e_b\}\})=0$. Hence $g(\hat x_k)=0$ for all $n$ $\pp_{\hat x}\as$ and the proposition holds.
\end{proof}

The place dependent IFS defined by $p_i:\hat x\mapsto \|A_ix\|^2$ and $f_i:\hat x\mapsto A_i\cdot \hat x$ is not contracting as defined in \cite{barnsley_invariant_1988} either. Note that since $A_1$ and $A_2$ are invertible, $f_1$ and $f_2$ are well defined on $\mathrm{P}(\cc^2)$.
\begin{prop}
    The place dependent IFS defined by the Markov kernel $\Pi$ is not contracting. Namely,
    $$\sup_{\hat x, \hat y}\frac{\ee_{\hat x}(d(V_1\cdot \hat x,V_1\cdot \hat y))}{d(\hat x,\hat y)}\geq 1.$$
\end{prop}
\begin{proof}
    Take $\hat x=\hat e_a$ and $\hat y=\hat e_b$, then $d(\hat x,\hat y)=1$ and $A_1\cdot \hat x=\hat x$, $A_1\cdot \hat y=\hat y$, $A_2\cdot \hat x=\hat y$ and $A_2\cdot \hat y=\hat x$. Hence, $\frac{\ee_{\hat x}(d(V_1\cdot \hat x,V_1\cdot \hat y))}{d(\hat x,\hat y)}=1$ and the proposition is proved.
\end{proof}

For this example, the variance $\gamma_g^2$ takes a simple form.
\begin{prop}
    For any $g:\sta\to\cc$ Hölder continuous, $$\gamma_g^2=\ee_{\nu_{inv}}\left(\left(g(\hat x)-\ee_{\nu_{inv}}(g(\hat x))\right)^2\right)=p g^2(\hat e_a)+q g^2(\hat e_b)-(pg(\hat e_a)+q g(\hat e_b))^2.$$
\end{prop}
\begin{proof}
    Since $\operatorname{supp}\nu_{inv}=\{\hat e_a,\hat e_b\}$, it is sufficient to compute $\tilde g(\hat e_a)$ and $\tilde g(\hat e_b)$. Direct computation shows there exists $\tilde g$ solution of the Poisson equation~\eqref{eq:Poisson eq} such that $\tilde g(\hat e_a)=g(\hat e_a)-\nu_{inv}(g)$ and $\tilde g(\hat e_b)=g(\hat e_b)-\nu_{inv}(g)$. Indeed, $\Pi g(\hat e_a)=\Pi g(\hat e_b)=\nu_{inv}(g)$. Moreover, $\nu_{inv}((\Pi \tilde g)^2)=0$. Hence, $\gamma_g^2=\nu_{inv}(\tilde g^2)=\ee_{\nu_{inv}}(g(\hat x))$.
\end{proof}

\paragraph{\bf Acknowledgements} The authors were supported by the ANR project ``ESQuisses", grant number ANR-20-CE47-0014-01 and by the ANR project ``Quantum Trajectories'' grant number ANR-20-CE40-0024-01. C.P.is also supported by the ANR projects Q-COAST ANR-19-CE48-0003. C.P. and T.B. are also supported by the program ``Investissements d'Avenir'' ANR-11-LABX-0040 of the French National Research Agency. 

% \bibliographystyle{alpha}

% \bibliography{QTraj_abbrv}
\newcommand{\etalchar}[1]{$^{#1}$}

\end{document}